\documentclass[11pt]{amsart}

\usepackage[T1]{fontenc}
\usepackage[utf8]{inputenc}
\usepackage{lmodern}
\usepackage{amsmath,amssymb,amsthm,mathtools}
\usepackage{mathrsfs}
\usepackage{placeins}
\usepackage{microtype}
\usepackage[colorlinks=true,linkcolor=blue,citecolor=blue,urlcolor=blue]{hyperref}
\usepackage[nameinlink,capitalise,noabbrev]{cleveref}
\usepackage[a4paper,margin=1in]{geometry}
\usepackage{tikz}

\usepackage{tikz-cd}

\usetikzlibrary{arrows.meta,positioning,calc}

\newtheorem{theorem}{Theorem}[section]
\newtheorem{proposition}[theorem]{Proposition}
\newtheorem{lemma}[theorem]{Lemma}
\newtheorem{corollary}[theorem]{Corollary}
\theoremstyle{definition}

\newtheorem{question}[theorem]{Question}
\theoremstyle{remark}
\newtheorem{remark}[theorem]{Remark}

\newcommand{\CC}{\mathbb{C}}
\newcommand{\RR}{\mathbb{R}}

\title[Constructive Peter--Weyl Theory]{Constructive Peter--Weyl Theory: What is Known and What Remains Open}
\author{Takao Inou\'{e}}
\address{Faculty of Informatics, Yamato University, Osaka, Japan}
\email{inoue.takao@yamato-u.ac.jp}
\thanks{Personal Email: takaoapple@gmail.com. I prefer my personal email address for correspondence.}
\date{October 12, 2026}


\begin{document}

\begin{abstract}
What is known about constructive Peter--Weyl theory, and what remains open? This expository article separates established theorems, comparisons between formulations, and questions requiring further investigation in the Bishop--Coquand--Spitters line. Its central finding is that constructive Peter--Weyl theory already has an established compact-group approximation core: under specified analytic and Gelfand hypotheses, finite-rank central convolution projections converge strongly in the Haar $L^2$ norm. We make these hypotheses and the projection-kernel normalization explicit, and record the resulting spectral Hilbert completion. This answers the question of completion affirmatively for that precise core theorem. Broader claims about the whole classical package or its validity across constructive foundations require separate comparisons; uniform approximation, irreducible representation data, and point-set recovery are not inferred from strong convergence alone. The historical problems of left--right almost periodicity and universal Bohr compactification are examined with a dated follow-up literature review, distinguishing related progress from a verified resolution. A sufficient metric realization criterion, finite examples, and a separate nonassociative outlook clarify the scope of the established method. An extended appendix proves the classical Peter--Weyl theorem through Haar averaging, convolution, spectral density, point separation, and irreducible orthogonality. Stage-by-stage comparisons explain which arguments transfer constructively and which require additional data or a different proof. The article therefore identifies what is proved and what additional statements or data remain to be examined, without treating an unidentified resolution as proof that a problem is currently open.
\end{abstract}

\maketitle

\tableofcontents

\section{Introduction}

The purpose of this article is the question expressed in its title: \emph{what is known about constructive Peter--Weyl theory, and what remains open?} Answering it requires a distinction between the established compact-group theorem, comparisons of its formulations, and the surrounding questions that affect its reach. The question ``is constructive Peter--Weyl theory complete?'' is assessed on the basis of this distinction.

\paragraph{What is already known.}
Coquand and Spitters have given a constructive Peter--Weyl proof \cite{CoquandSpittersPW}. The existence of the compact-group approximation theorem is therefore settled in their constructive setting. In the compact metric formulation reviewed here, finite-rank central convolution projections approximate every vector in the Haar $L^2$ space, with explicit integration, separability, normability, and Gelfand input. We specify the kernel normalization and the error norm, and prove the spectral Hilbert-sum consequence. This established content should be clearly visible before the paper turns to questions or extensions.

\paragraph{What remains to be distinguished.}
The classical Peter--Weyl package includes uniform coefficient approximation, Hilbert decomposition, irreducible representations, and separation. Constructively, a statement about one part does not establish every other part in every foundation. The comparisons below identify which implications are proved here and which representation data or foundational transfers need to be checked separately. They do not infer that a formulation is unknown merely because it is absent from the theorem being reviewed. In particular, the blockwise Hilbert completion follows from the stated projections and is not listed as a missing result.

\paragraph{Historical questions and their present evidential status.}
The left--right almost-periodicity problem raised by Spitters \cite{SpittersAP} and the Gelfand-based Bohr compactification problem raised by Coquand and Spitters \cite{CoquandSpittersGelfand} belong to the scope of the survey. We state their hypotheses and examine later related work. Our dated search did not identify a resolution in the precise settings formulated here, but that finding is not a proof of present unresolvedness. Accordingly, historical attribution, verified later progress, and proposed tasks are kept distinct. This is essential to interpreting the words ``remains open'' in the title responsibly.

\paragraph{Scope and overall assessment.}
``Constructive'' refers primarily to Bishop-style analysis \cite{BishopBridges}; the compact metric formulation allows countable dependent choice. Locale theory and constructive Gelfand representation are considered where they help compare foundations. The scope is not an exhaustive survey of every school of constructivism. Within this scope, the answer about completion is affirmative for the stated compact-group approximation core. A wider completeness claim must specify the additional classical formulations, hypotheses, and foundational invariance that it intends to include. Questions about general-group compactification are relevant to that wider assessment, while extensions to loops and quasigroups concern an enlarged class of objects.

The contribution is expository and comparative. The precise theorem statements support the assessment of what is established; a sufficient metric realization criterion and an $S_3$ counterexample clarify the compactification step. The later nonassociative section gives an algebraic obstruction and a conditional classical approximation criterion, followed by constructive tasks. These results and questions support the paper's main status assessment. The classical appendix provides a detailed reference proof and a stage-by-stage comparison with constructive methods.

\subsection{Status of the statements}

The numbered statements have the following roles.
\begin{enumerate}
\item \Cref{thm:CS-schematic} formulates the Coquand--Spitters approximation result with explicit analytic input and kernel normalization. Its constructive density input is cited; the accompanying argument gives the operator consequences and completion step.
\item \Cref{prop:finite-stage,prop:coefficient-viewpoint} are elementary operator consequences. \Cref{prop:compactification-principle} proves a sufficient metric realization criterion; \Cref{cor:AP-method,prop:three-step} give the conditional pullback consequences.
\item Section~\ref{sec:loops} proves an algebraic obstruction and a conditional classical transformation-group criterion, then specifies the remaining constructive tasks.
\item The appendix proves the classical package and compares its proof stages with constructive requirements. Its classical spectral and approximation arguments do not replace the constructive density input in the main text.
\end{enumerate}

\subsection{Structure of the note}

\Cref{sec:classical-package,sec:why-reformulation} recall the classical package and explain the constructive distinctions. \Cref{sec:known-results} gives the compact-group approximation statement, metric realization criterion, and Gelfand background. \Cref{sec:dictionary} compares formulations, and \cref{sec:established} summarizes what is established and distinguishes the questions requiring further examination. \Cref{sec:open} separates historical questions, comparisons with an established baseline, and proposed tasks. \Cref{sec:loops} develops the conditional nonassociative outlook; \cref{sec:program,sec:conclusion} state the further program and conclusions. The appendix has a dual purpose. It proves the classical Peter--Weyl package in six detailed stages, beginning with Haar averaging and continuing through compact convolution operators, finite-dimensional invariant spaces, $L^2$ density, point separation and uniform density, and irreducible orthogonality. After each stage, a constructive comparison identifies the available finite or quantitative data, the classical existence principle used, and the corresponding role of the main-text formulation. The reader can thus distinguish the algebraic and integral calculations that persist from the spectral, topological, and representation-theoretic steps that need additional hypotheses or a reorganized proof. The appendix is part of the evidence for the paper's status assessment, not merely an appended classical theorem.

\section{The Classical Peter--Weyl Package}\label{sec:classical-package}

For a compact Hausdorff group $G$, the phrase ``the Peter--Weyl theorem'' often refers not to a single statement but to a package of closely related statements.
Among the most common are the following.

\begin{enumerate}
  \item Matrix coefficients of finite-dimensional continuous unitary representations form a $*$-subalgebra of $C(G)$ that is uniformly dense.
  \item The irreducible characters span a dense subspace of the central continuous functions.
  \item The left regular representation on $L^2(G)$ decomposes as a Hilbert direct sum of finite-dimensional irreducible subrepresentations.
  \item Convolution by suitable central functions yields finite-rank approximation operators.
\end{enumerate}

These formulations are equivalent or closely intertwined in the classical setting, but they are not equally suitable constructively.
Some are naturally pointwise and spectral, some are naturally Hilbertian, and some are better phrased in terms of operator approximations or characters.
Thus one of the first tasks in constructive Peter--Weyl theory is to distinguish the classical package into separate components.

\begin{remark}
For later purposes it is useful to think of Peter--Weyl theory as consisting of at least four ingredients: compactness, invariant integration, finite-dimensional approximation, and representation-theoretic separation.
A later generalization to loops or quasigroups may preserve some of these ingredients while losing others.
\end{remark}

\section{Why Constructive Reformulation Is Needed}\label{sec:why-reformulation}

\begin{remark}[Foundational stance]\label{rem:foundational-stance}
Throughout this note, ``constructive'' should be read primarily in the Bishop-style sense of \cite{BishopBridges}, as explained in the introduction, with auxiliary use of locale theory and formal topology when constructive compactness or constructive Gelfand representation is discussed.
Thus the present note is not intended as a survey of all forms of constructivism.
In particular, Russian recursive mathematics, proof-mining approaches, Weihrauch-theoretic computable analysis, and type-theoretic formalizations are not treated here as parallel frameworks on an equal footing, even though they may eventually provide useful further perspectives on Peter--Weyl-type phenomena.
\end{remark}

In classical expositions, the Peter--Weyl theorem may be reached through compact self-adjoint operators, spectral theory, maximality arguments, or representation-theoretic arguments involving irreducibility and orthogonality.
Constructively, such arguments are not automatically available in the same form.
Even when the final theorem remains valid, the route to it may have to be reorganized.

A second issue is that the classical theorem is often stated in a manner tailored to completed Hilbert spaces and existence theorems for irreducible representations.
Constructive mathematics tends to prefer formulations with clearer computational or approximation content.
This is one reason why constructive versions of Peter--Weyl often emphasize approximation by explicitly described finite-rank operators, characters, almost periodic functions, or compact completions associated with a given function.

A third issue is methodological.
Constructive mathematics frequently reveals that a classical theorem really contains several distinct layers.
One layer may be approximation-theoretic, another topological, another spectral, and another dependent on a specific choice principle.
By forcing these layers to be separated, the constructive viewpoint often makes the structure of a theorem more transparent even for readers whose primary interests are classical.
In this sense, constructive reformulation is not merely a restriction; it can also be a conceptual sharpening.

\section{Known Constructive Results}\label{sec:known-results}

The available literature suggests that constructive Peter--Weyl theory should not be described as nonexistent.
Rather, its core already appears in a definite compact-group form, and its surrounding theory extends outward through almost periodic functions, constructive Gelfand theory, and locale-theoretic compactness.
The point of the present section is to display these strands in a somewhat theorem-oriented manner.

\subsection{Compact groups}

\paragraph{Analytic and foundational conventions.}
For the point-set result in this subsection, work in Bishop-style analysis with countable dependent choice. A compact metric space means a complete, totally bounded metric space; total boundedness is supplied by finite $\varepsilon$-nets for every $\varepsilon>0$. Let $G$ be such a space, with a group structure whose multiplication and inversion are uniformly continuous. Write $C(G)$ for uniformly continuous complex-valued functions with the supremum norm. Fix normalized left- and right-invariant Haar integration, write
\[
\langle f,h\rangle=\int_G f(x)\overline{h(x)}\,d\mu(x),
\qquad \|f\|_2^2=\langle f,f\rangle,
\]
and let $L^2(G)$ be the corresponding Hilbert completion. Inner products are linear in the first variable. These conventions specify the metric compact-group setting; they do not assert the same point-set theorem for every compact Hausdorff group or compact locale.

We make explicit the operator-algebra input to the point-set Gelfand argument. With
\[
(f*h)(x)=\int_G f(y)h(y^{-1}x)\,d\mu(y),
\qquad f^*(x)=\overline{f(x^{-1})},
\]
let $Z$ be the center of the continuous convolution algebra. Let $A_Z$ be the unital commutative $C^*$-algebra generated on $L^2(G)$ by the identity and the operators $T_zf=z*f$, $z\in Z$. For the formulation below, assume $L^2(G)$ and $A_Z$ are separable, that each element of $A_Z$ has a located real operator norm, and that the point-set constructive Gelfand representation applies to $A_Z$. Normability means this locatedness; it is not merely the existence of an upper bound for the operator norm. These assumptions identify the Gelfand input being used, rather than hide it inside the phrase ``constructive setting.'' The convolution compactness and normability arguments in \cite{CoquandSpittersPW} are part of the justification for this input in its Bishop-style treatment. No claim of a choice-free point-set formulation is made here.

\paragraph{Spectral characters and convolution kernels.}
Let $X$ be the compact metric Gelfand spectrum of $A_Z$. The spectral character space relevant to convolution is
\[
\Sigma=\{\sigma\in X:\text{there exists }z\in Z
\text{ with }|\sigma(T_z)|>0\}.
\]
Thus a spectral character $\sigma$ is a multiplicative $*$-homomorphism on the operator algebra, not a function on $G$. The central spectral construction gives normalized central functions $\psi_\sigma\in Z$ with $\psi_\sigma(e)=1$. We use the compatible indexing for which
\[
a_\sigma=\|\psi_\sigma\|_2^2>0,\qquad
\psi_\sigma*\psi_\sigma=a_\sigma\psi_\sigma,
\qquad \psi_\sigma^*=\psi_\sigma.
\]
The \emph{idempotent convolution kernel} is
\[
p_\sigma=\frac{\psi_\sigma}{a_\sigma},\qquad
p_\sigma*p_\sigma=p_\sigma,\qquad p_\sigma^*=p_\sigma.
\]
Kernels for distinct spectral points satisfy $p_\sigma*p_\tau=0$. This normalization is fixed by convolution idempotence, rather than by requiring $\|p_\sigma\|_2=1$. Centrality implies $p_\sigma*f=f*p_\sigma$, so these operators also commute with the left translations used in the appendix.

\begin{theorem}[Coquand--Spitters, $L^2$ approximation in idempotent-kernel normalization]\label{thm:CS-schematic}
Under the compact metric and Gelfand hypotheses just specified, the central spectral construction provides a discrete character space $\Sigma$ and kernels $p_\sigma$ as above. For a finite, explicitly listed set $F\subset\Sigma$ of distinct spectral points, define
\[
P_Ff=\sum_{\sigma\in F}p_\sigma*f.
\]
Then $P_F$ extends to a finite-rank orthogonal projection on $L^2(G)$ and maps $L^2(G)$ into $C(G)$. For every $f\in C(G)$ and every $\varepsilon>0$, there is such an $F_0$ for which
\[
\|f-P_Ff\|_2<\varepsilon
\quad\text{for every finite }F\supseteq F_0.
\]
The same conclusion holds for every $f\in L^2(G)$ by completion. In particular, the asserted convergence is strong convergence on $L^2(G)$, not convergence in operator norm or an assertion that these partial sums converge uniformly on $G$.
\end{theorem}

\begin{proof}[Source and normalization of the result]
The approximation conclusion is the constructive Peter--Weyl result of \cite{CoquandSpittersPW}; the constructive positive-central expansion and Plancherel argument supply its finite-tail estimate. We express the spectral terms using idempotent kernels $p_\sigma$. The identities above make $f\mapsto p_\sigma*f$ mutually orthogonal projections. Each is compact by convolution compactness. A compact orthogonal projection has finite-dimensional range, and the continuous-kernel estimate makes its range a subspace of $C(G)$.

Finite sums therefore give finite-rank orthogonal projections of norm at most one. Their orthogonality gives, for $F_0\subseteq F$,
\[
\|f-P_{F_0}f\|_2^2
=\|f-P_Ff\|_2^2+\|(P_F-P_{F_0})f\|_2^2.
\]
Thus a finite-tail estimate stays valid under refinement. Finally, for $f\in L^2(G)$ choose $h\in C(G)$ close to $f$; then
\[
\|f-P_Ff\|_2\le2\|f-h\|_2+\|h-P_Fh\|_2.
\]
This proves the stated extension without any assertion of uniform convergence. The elementary operator consequences and completion step are spelled out here; the constructive density input is cited, not reproved by the classical appendix.
\end{proof}

\begin{remark}[What the operator formulation records]
The theorem identifies the finite approximation data, ambient Hilbert space, and error norm. It does not require a global classification of all irreducible representations as its starting point. Nor does it conclude that an individual finite-dimensional projection range is an irreducible representation: classically a central projection selects an entire isotypic coefficient space.
\end{remark}

\begin{proposition}[Finite-stage approximation and error monotonicity]\label{prop:finite-stage}
In the setting of \Cref{thm:CS-schematic}, set
\[
V_F=\bigoplus_{\sigma\in F}\operatorname{Ran}P_{\{\sigma\}}.
\]
Then $V_F=\operatorname{Ran}P_F$ is finite-dimensional, $P_Ff\in V_F$, and finite refinement decreases the $L^2$ error. No analogous monotonicity in the supremum norm is asserted.
\end{proposition}
\begin{proof}
The ranges are mutually orthogonal, and the Pythagorean identity in the preceding proof gives the error inequality.
\end{proof}

\begin{remark}[Constructive emphasis of the compact-group theory]\label{rem:compact-emphasis}
The formulation foregrounds finite-rank approximation, central spectral data, and a specified constructive Gelfand input. This is an interpretation of the proof architecture, rather than a further mathematical theorem. It should be distinguished from the stronger task of comparing alternative foundational settings.
\end{remark}

\paragraph{Classical character dictionary and a normalization check.}
Classically, identify the spectral index with an irreducible representation using a compatible Fourier convention, and write
\[
c_\pi(x)=\operatorname{Tr}\pi(x),\qquad
\psi_\pi=c_\pi/d_\pi,\qquad d_\pi=\dim H_\pi.
\]
Schur orthogonality gives
\[
\|c_\pi\|_2=1,\qquad
\|\psi_\pi\|_2^2=d_\pi^{-2},\qquad
p_\pi=d_\pi c_\pi.
\]
Thus $c_\pi$ is an $L^2$-normalized class function, whereas $p_\pi$ is the convolution projection kernel. Their roles differ when $d_\pi>1$. The associated coefficient space has dimension $d_\pi^2$, not $d_\pi$. These classical identifications clarify the normalization; they are not being used to import a classical irreducibility argument into the constructive theorem.

For the two-dimensional standard representation of $S_3$, the trace character has values $2,0,-1$ on the identity, transpositions, and three-cycles. With normalized counting measure, $\psi=c/2$ has squared $L^2$ norm $1/4$, and $p=2c$ has values $4,0,-2$. Direct convolution gives $p*p=p$ and $c*c=c/2$. The projection has rank four. This finite example checks both the dimension factor and the distinction between a unit-norm character and an idempotent kernel.

\paragraph{$L^2$ partial sums versus uniform approximation.}
Already on the circle group, ordinary Fourier partial sums of a continuous function need not converge uniformly. On an infinite-dimensional $L^2(G)$, no finite-rank orthogonal projection approaches the identity in operator norm: its orthogonal complement gives $\|I-P_F\|=1$. Neither phenomenon contradicts \Cref{thm:CS-schematic}.

Uniform approximation can instead use smoothing or summation kernels. On the circle, Fej\'{e}r averages are finite-rank convolution operators with positive kernels. More generally, for a normalized approximate-identity kernel $u$ and a finite coefficient kernel $q$, the elementary estimate
\[
\|q*f-f\|_\infty
\le\|q-u\|_1\|f\|_\infty+\|u*f-f\|_\infty
\]
separates kernel approximation from smoothing. This estimate is not a claim that suitable $q$ are constructively available under arbitrary hypotheses. Such a uniform approximation theorem needs its own constructive input and should be assigned a different family of operators from the orthogonal projections $P_F$.

\subsection{Almost periodic functions and compact completions}

Spitters uses almost periodic functions to reduce approximation questions to compact groups \cite{SpittersAP}. For a noncommutative group, however, a metric completion becomes a \emph{group} completion only when its metric identification and group operations are compatible. We give an explicit sufficient formulation, separating the elementary completion step from the analytic hypotheses needed for Peter--Weyl approximation. We continue to allow countable dependent choice as in the preceding subsection.

\paragraph{Translation data and almost periodicity.}
Let $G$ be a topological group and $f:G\to\CC$ a bounded continuous function. Set
\[
(\ell_s f)(x)=f(sx),\qquad (r_t f)(x)=f(xt).
\]
The left translation family is totally bounded in the uniform sense if, for every $\varepsilon>0$, there are $s_1,\ldots,s_n\in G$ such that every $\ell_s f$ is within $\varepsilon$ of one of these translates at every point of $G$. Define right total boundedness similarly. These finite-cover statements use the pointwise seminorms $h\mapsto|h(x)|$ and do not presume that arbitrary bounded continuous functions have a located supremum norm. Call $f$ left or right almost periodic according to these conditions, and call it almost periodic when both hold. We do not infer one direction from the other.

\paragraph{Why a one-sided metric needs a check.}
The displayed construction in Section~4 of \cite{SpittersAP} uses the pseudometric
\[
d_f^{\ell}(a,b)=\sup_{s\in G}|f(sa)-f(sb)|.
\]
Left almost periodicity makes this supremum a located real: approximate the translation family by finitely many translates, evaluate at $a,b$, and take the maximum of their finitely many differences. A uniform error $\varepsilon$ in the translates gives an error at most $2\varepsilon$ in these values. Reindexing $s$ shows left invariance of $d_f^{\ell}$. Right invariance, and hence compatibility with multiplication on its zero-distance classes, does not follow from this formula.

For example, take $G=S_3$, $H=\{e,(12)\}$, and $f=\mathbf1_H$. This function is both left and right almost periodic, since the group is finite. For $h=(12)$, $f(sh)=f(s)$ for every $s$, so $d_f^{\ell}(e,h)=0$. But for $c=(23)$, choosing $s=c^{-1}$ gives
\[
f(sc)=1,\qquad f(shc)=f(c^{-1}hc)=0,
\qquad d_f^{\ell}(c,hc)=1.
\]
Thus right multiplication by $c$ is not well-defined on the zero-distance classes. The metric identification is the three-point coset space $G/H$, not a quotient group, and is already complete. This example checks the specific one-sided construction; it does not refute a Bohr approximation theorem formulated using a valid compact-group realization.

\begin{proposition}[Metric compact-group realization: explicit sufficient hypotheses]\label{prop:compactification-principle}
Let $G$ be a topological group and $f:G\to\CC$ a bounded continuous function. Either of the following sets of data suffices for a compact metric group realization of $f$:
\begin{enumerate}
\item $G$ is Abelian and $f$ has a totally bounded translation family in the uniform sense just defined. Use
\[
\delta_f(a,b)=d_f(a,b):=\sup_{x\in G}|f(ax)-f(bx)|.
\]
\item $G$ is arbitrary and the quantities
\[
\delta_f(a,b)=\rho_f(a,b):=\sup_{s,t\in G}|f(sat)-f(sbt)|
\]
are supplied as located real suprema. In addition, the pseudometric space $(G,\rho_f)$ has finite $\varepsilon$-nets with centers in $G$ for every $\varepsilon>0$, and for every $a\in G$ and $\varepsilon>0$ there is a neighborhood $U$ of $a$ such that $\rho_f(a,b)<\varepsilon$ for $b\in U$.
\end{enumerate}
Then the metric identification of $(G,\delta_f)$ is a group, its completion $H_f$ is a compact metric group, and there are a continuous homomorphism $q_f:G\to H_f$ with dense image and a unique continuous function $\widetilde f\in C(H_f)$ satisfying
\[
f=\widetilde f\circ q_f.
\]
The function $\widetilde f$ is $1$-Lipschitz for the completed metric. No decidability assumption on membership in the zero-distance kernel is required.
\end{proposition}

\begin{proof}
In case (1), locatedness of $d_f$ follows by evaluating a finite approximation to the translation family, as above, using commutativity to write $f(ax)=f(xa)$. Choose $a_1,\ldots,a_n$ approximating all translates pointwise within $\varepsilon/2$. Then $d_f(a,a_i)\le\varepsilon/2<\varepsilon$ for some $i$ for every $a$. These same centers therefore give a finite metric net. Commutativity and reindexing show that $d_f$ is invariant under both left and right multiplication. The quotient map is continuous: a finite approximation to the translates bounds $d_f(a,e)$ by
\[
\max_i|f(s_i a)-f(s_i)|+2\eta
\]
for arbitrarily small $\eta>0$. Each of the finitely many terms is continuous at $e$. Bi-invariance then gives continuity at every point.

In case (2), the supremum is part of the supplied data. Its triangle inequality follows pointwise. Reindexing $s$ and $t$ gives
\[
\rho_f(ua,ub)=\rho_f(a,b)=\rho_f(av,bv)
\qquad(u,v\in G).
\]
Total boundedness and continuity of the quotient map are the explicit remaining assumptions of that case; they have not been inferred by selecting representatives from arbitrary subsets of finite covers.

In either case $\delta_f$ is a bi-invariant pseudometric. Consequently
\[
\delta_f(ab,a'b')\le\delta_f(a,a')+\delta_f(b,b'),
\qquad
\delta_f(a^{-1},b^{-1})=\delta_f(a,b).
\]
The relation $a\sim b$ defined by $\delta_f(a,b)=0$ is therefore a group congruence. Equivalently, $N_f=\{a:\delta_f(a,e)=0\}$ is normal and the metric identification is $G/N_f$, with equality understood through zero distance. Multiplication and inversion descend to this identification and extend uniquely to its metric completion. The displayed estimates make their extensions uniformly continuous, and the group laws hold by density. The completion is complete and totally bounded, hence compact in the Bishop sense.

Finally $|f(a)-f(b)|\le\delta_f(a,b)$, by choosing identity translation parameters. Thus $f$ descends to a $1$-Lipschitz function on the metric identification and extends uniquely to $\widetilde f$ on $H_f$. Continuity of $q_f$ with respect to the original topology was verified in case (1) and assumed in case (2). Its homomorphism property and dense image follow from the construction.
\end{proof}

\begin{corollary}[Transfer of the compact-group $L^2$ approximation]\label{cor:AP-method}
For a realization $(H_f,q_f,\widetilde f)$ obtained in \Cref{prop:compactification-principle}, assume in addition that $H_f$ satisfies the Haar integration, separability, normability, and Gelfand hypotheses of \Cref{thm:CS-schematic}. Then
\[
\widetilde f_F=P_F^{H_f}\widetilde f\longrightarrow\widetilde f
\quad\text{in }L^2(H_f),
\qquad
f_F=q_f^*\widetilde f_F,
\]
where $q_f^*h=h\circ q_f$. Each $f_F$ belongs to a finite-dimensional space of continuous functions on $G$ obtained by pullback. Convergence is measured by the Haar norm on $H_f$, not by an unspecified Haar measure on $G$ and not by the supremum norm on $G$.
\end{corollary}

\begin{proof}
Apply \Cref{thm:CS-schematic} to $H_f$ and $\widetilde f$. Composition with $q_f$ preserves finite-dimensionality. Density of $q_f(G)$ makes $q_f^*:C(H_f)\to C_b(G)$ injective and preserves supremum norms on this image. On $\mathcal B_f=q_f^*C(H_f)$ define
\[
\|q_f^*h\|_{2,f}:=\|h\|_{L^2(H_f)}.
\]
Injectivity makes this well-defined, and the approximation says exactly $\|f_F-f\|_{2,f}\to0$. A supremum-norm approximation would require an additional uniform approximation theorem on $H_f$.
\end{proof}

\paragraph{Scope of the general-group assertion.}
Case (2) is a sufficient criterion, not a claim that left and right almost periodicity alone constructively supplies all of its data. Locatedness of the two-sided supremum, finite nets with actual group centers, and continuity into the new metric must be justified for the intended class of examples. In the $S_3$ example all these conditions can be checked directly; $\rho_f(a,b)=1$ for distinct $a,b$, so the completion is the six-element group itself. The defective one-sided quotient and the valid two-sided quotient are thereby distinguished explicitly.

The function-dependent realization $H_f$ is also distinct from the universal Bohr compactification $bG$. The latter requires a universal property simultaneously involving suitable compact-group homomorphisms. No such universal property follows merely from constructing $H_f$ for one function. Similarly, a compact locale is not being used as an interchangeable notation for this compact metric group.

Spitters's discussion also highlights that classical closure properties of almost periodic functions require care constructively, and that the passage from left to right almost periodicity was not resolved in the cited paper. We keep this as a historically attributed question, with the dated follow-up review in Section~\ref{sec:historical-followup} distinguishing related progress from a verified resolution. In the present subsection, the compactification criterion and the conditional approximation corollary are the precise statements used later.

\subsection{Locales and constructive Gelfand theory}

A third strand comes from constructive Gelfand representation and locale theory \cite{CoquandSpittersGelfand,CoquandFormalTopology,BanaschewskiMulvey}.
Here compact spaces are often replaced by compact locales, and representation theorems naturally live in a point-free setting.
This matters because the proof strategy in \cite{CoquandSpittersPW} is tied to constructive Gelfand theory, and because locale-theoretic compactness often provides the right constructive substitute when point-set compactness is too strong or too choice-dependent.

From this perspective, one of the key lessons is methodological.
Constructive Peter--Weyl theory is entangled with constructive spectral theory.
One does not simply ``repeat the classical proof more carefully.''
Instead, one often proves a representation theorem first in a point-free form and only then asks under what additional hypotheses it can be read as an ordinary theorem about spaces of points, norms, or concrete compact groups.
This is exactly the kind of distinction that later becomes crucial if one hopes to move beyond groups.

\begin{remark}[Point-free background and its scope]\label{rem:pointfree-background}
Constructive Gelfand representation and locale theory provide background for comparing point-set and point-free formulations \cite{CoquandSpittersGelfand,CoquandFormalTopology,BanaschewskiMulvey}. This contextual observation is not a Peter--Weyl theorem for arbitrary compact locales. The point-set statement in \Cref{thm:CS-schematic} uses its own metric, separability, normability, and Gelfand assumptions; none is replaced here by the unqualified assertion that a locale is compact. Passing from a point-free representation to a space of points is a separate task requiring a precise representation theorem and foundational assumptions.
\end{remark}

\section{A Working Dictionary: Classical vs. Constructive}\label{sec:dictionary}

This section records a provisional dictionary between common classical formulations and their constructive counterparts.
The point is not that each classical statement has a perfect constructive twin, but rather that the role played by a statement in classical theory is often taken over by a nearby statement in the constructive setting.
What was expressed informally in earlier sections can now be summarized in a slightly more formula-oriented way.

\begin{center}
\begin{tabular}{p{0.43\textwidth}p{0.48\textwidth}}
\hline
\textbf{Classical formulation} & \textbf{Constructive tendency} \\
\hline
Irreducible unitary representations classify finite-dimensional pieces & Use characters, finite-rank approximation operators, or explicit approximation statements \\
Matrix coefficients are dense in $C(G)$ & Density may be reformulated through approximants built from constructive representation data \\
Hilbert direct-sum decomposition of $L^2(G)$ & Approximation in $L^2$ or in terms of finite-rank operators is often more primary \\
Compact Hausdorff space of points & Compact locale may appear before enough points are available \\
Global theorem for a topological group & Construct a compatible compact group realization; verify its analytic hypotheses before applying approximation \\
\hline
\end{tabular}
\end{center}

The first three rows may be compressed into the following symbolic comparison.
Classically, one writes the Peter--Weyl approximation package as
\[
L^2(G)=\widehat{\bigoplus}_{\pi\in \widehat G}\mathcal C_\pi,
\qquad
f \sim \sum_{\pi\in \widehat G} f_\pi,
\qquad
P_F(f)=\sum_{\pi\in F} f_\pi,
\]
where $\mathcal C_\pi\cong H_\pi\otimes H_\pi^*$ is the coefficient space of dimension $d_\pi^2$, and $F$ ranges over finite subsets of the unitary dual.
Constructively, the decomposition display is typically replaced by the approximation display
\[
P_F(f)=\sum_{\sigma\in F} p_\sigma * f \longrightarrow f,
\]
Here $F$ is finite, explicitly listed spectral data as in \Cref{thm:CS-schematic}, and the displayed limit is in $L^2(G)$. The formula does not assert uniform convergence on $G$.

\begin{remark}[How to read the dictionary]\label{rem:dictionary}
The table compares mathematical roles and proof strategies. Its entries are not equivalences proved in this article. The precise $L^2$ statement is \Cref{thm:CS-schematic}, its finite-stage consequence is \Cref{prop:finite-stage}, and the sufficient group-completion criterion is \Cref{prop:compactification-principle}. Point-free background has the separate status explained in \Cref{rem:pointfree-background}. Thus the analogy between decomposition, approximation, and compactification should not be used as a premise for a further theorem without checking the relevant hypotheses.
\end{remark}

\begin{proposition}[Range of a finite spectral projection]\label{prop:coefficient-viewpoint}
Under the hypotheses of \Cref{thm:CS-schematic}, for finite spectral data $F$ define
\[
V_F=\operatorname{span}\{p_\sigma*f:\sigma\in F,\ f\in L^2(G)\}.
\]
Then
\[
V_F=\operatorname{Ran}P_F
=\bigoplus_{\sigma\in F}\operatorname{Ran}P_{\{\sigma\}}.
\]
In particular $V_F$ is finite-dimensional and $P_F$ is its orthogonal projection. Classically, let $\widehat F$ denote the finite set of irreducible classes corresponding to $F$ under the character identification in Section~\ref{sec:known-results}. Then $V_F=\bigoplus_{\pi\in\widehat F}\mathcal C_\pi$ and $\dim V_F=\sum_{\pi\in\widehat F}d_\pi^2$.
\end{proposition}
\begin{proof}
Each generator $p_\sigma*f$ is fixed by $P_F$, since the spectral projections are mutually orthogonal. Thus $V_F\subseteq\operatorname{Ran}P_F$. Conversely $P_Ff=\sum_{\sigma\in F}p_\sigma*f$ belongs to $V_F$. Orthogonality gives the direct sum, and finite-dimensionality follows from \Cref{thm:CS-schematic}. The last assertion is classical: the character dictionary identifies these summands with the coefficient spaces, and Schur orthogonality in the appendix gives their dimensions.
\end{proof}

\begin{remark}
The dictionary is therefore not a translation glossary in the strict logical sense.
It is a guide to which constructive statements should be regarded as carrying the mathematical burden of the corresponding classical ones.
In particular, the replacement
\[
\text{``global decomposition''} \rightsquigarrow \text{``directed family of explicit approximants''}
\]
is not merely terminological; it changes what counts as primary mathematical data.
\end{remark}

\section{What Is Known and What Remains to Be Clarified}\label{sec:established}

The results above support the following status assessment. The categories distinguish a proved conclusion from a comparison not completed here, a historically posed problem, and a proposed extension.
\begin{enumerate}
\item \textbf{Established compact-group approximation.} \Cref{thm:CS-schematic} gives finite-rank orthogonal convolution projections converging strongly on $L^2(G)$ under its stated Bishop-style and analytic hypotheses. Its constructive density input is an existing result of Coquand and Spitters. This core is an established theorem.
\item \textbf{Established spectral Hilbert completion.} \Cref{lem:spectral-sum}, proved below, identifies the Hilbert completion of the finite spectral blocks with $L^2(G)$. This is a consequence of the stated orthogonality and density, not an unresolved step between approximation and blockwise decomposition.
\item \textbf{Comparisons requiring separately specified content.} Irreducible block descriptions, uniform approximation, and point-set recovery involve additional statements and hypotheses. Sections~\ref{sec:dictionary} and~\ref{sec:open} describe what must be matched with existing constructive results. A comparison left unproved here is not thereby certified as an open problem.
\item \textbf{Historical problems with a dated review.} Questions~\ref{q:left-right-ap} and~\ref{q:bohr} record problems posed in 2005. Section~\ref{sec:historical-followup} identifies subsequent related progress and the limits of the search for same-setting resolutions. Their current unresolvedness is not asserted solely on the basis of that search.
\item \textbf{Proposed constructive tasks and extensions.} Two-sided metric data, joint realizations, and the loop and quasigroup targets specify data or results to investigate. They broaden the applications or objects of the theory and are distinguished from missing steps in the established compact-group proof.
\end{enumerate}

This assessment also answers the motivating question of completion: \emph{the stated constructive compact-group approximation core is established}. The broader theory is assessed statement by statement; an unqualified assertion of its completeness would require a specified wider package and evidence for each part. Conversely, the surrounding questions are not a reason to regard the already proved core as conjectural.

\begin{remark}[Scope of the established compact-group result]\label{rem:established-scope}
The roles of the preceding results can be summarized as follows.
\begin{enumerate}
  \item Constructive Peter--Weyl theory for compact groups is already genuinely present in the literature, rather than merely conjectural.
  \item \Cref{thm:CS-schematic} states a compact metric $L^2$ approximation result with a specified normable Gelfand input; finite-rank orthogonal projections play a central role.
  \item An almost periodic construction can use this result after a valid compact-group realization and its analytic hypotheses have been established separately.
  \item Constructive Gelfand representation is an explicit input to the point-set result; locale-theoretic comparisons have the separate status stated in \Cref{rem:pointfree-background}.
\end{enumerate}
These items describe a proof architecture, not an additional theorem deduced from interpretive statements. In particular, compactification and point-free transfer have their own hypotheses, which are not consequences of the $L^2$ approximation theorem.
\end{remark}

A slightly more structural summary is obtained by isolating the three recurring operations that appear throughout the literature:
\begin{enumerate}
\item Construct a finite-rank projection $P_F$ from finite spectral data.
\item Refine $F$, decreasing the $L^2$ error by orthogonality.
\item Use the finite-tail estimate to obtain $P_Ff\to f$ in $L^2(G)$.
\end{enumerate}
For translation data satisfying \Cref{prop:compactification-principle}, that criterion first supplies a compatible compact group $H_f$. The approximation pattern then applies on $H_f$ only after the separate hypotheses of \Cref{cor:AP-method} have been verified.

\begin{samepage}
\begin{proposition}[Finite approximation on a supplied compact realization]\label{prop:three-step}
Let $(H_f,q_f,\widetilde f)$ satisfy \Cref{cor:AP-method}. For finite spectral data $F$, form $\widetilde f_F=P_F^{H_f}\widetilde f$ and $f_F=q_f^*\widetilde f_F$. Refinement decreases $\|f-f_F\|_{2,f}$, and this error tends to zero. The identity realization gives the direct compact-group case.
\end{proposition}
\begin{proof}
The pullback norm is defined by the Haar $L^2$ norm on $H_f$. Thus the assertions follow from \Cref{prop:finite-stage,cor:AP-method}.
\end{proof}

\end{samepage}

\begin{corollary}[Established approximation core]
Under the hypotheses of \Cref{thm:CS-schematic}, finite central spectral projections approximate every vector in the compact-group $L^2$ space. Under the additional realization hypotheses of \Cref{cor:AP-method}, the same statement holds on the pullback space with its induced norm $\|\cdot\|_{2,f}$.
\end{corollary}
\begin{proof}
Apply \Cref{thm:CS-schematic} directly, or use \Cref{cor:AP-method} for the pullback assertion. No implication from $L^2$ convergence to uniform convergence is used.
\end{proof}


\begin{figure}[htbp]
  \centering
  \begin{tikzcd}[column sep=70pt, row sep=large]
  \widetilde f\in C(H_f)
    \arrow[r,"P_F^{H_f}"] \arrow[d,"q_f^*"']
  & \widetilde f_F\in V_F
    \arrow[d,"q_f^*"]
  \\
  f\in\mathcal B_f
    \arrow[r,"A_F"']
  & f_F\in q_f^*V_F
  \end{tikzcd}
  \caption{Approximation on a valid compact-group realization and pullback to the original group. The lift satisfies $f=\widetilde f\circ q_f$, and $\mathcal B_f=q_f^*C(H_f)$. The lower operator is defined by $A_F(q_f^*h)=q_f^*(P_F^{H_f}h)$. The upper approximation converges in $L^2(H_f)$; the lower one converges in the induced norm $\|\cdot\|_{2,f}$. Neither arrow denotes a group completion of a function, and no Haar measure on the original group or supremum-norm convergence is assumed. The direct compact-group route is recovered when $q_f$ is the identity.}
  \label{fig:constructive-pw-structure}
\end{figure}


\FloatBarrier

\section{Questions: Historical Context, Comparisons, and Proposed Tasks}\label{sec:open}

This section separates questions explicitly raised in the cited literature from comparisons of existing results and tasks proposed from the present formulation. The status distinctions are essential: a historical question is not automatically a currently open problem, and a task not carried out in this article is not automatically a new research problem. The primary sources identified below establish the historical attribution; a dated follow-up review below records the relevant developments and the limits of the search.

\begin{center}
\begin{tabular}{p{0.17\textwidth}p{0.35\textwidth}p{0.37\textwidth}}
\hline
Category & Evidence or purpose & Status asserted here \\
\hline
Literature-posed & Explicit discussions in the 2005 papers & Historically posed; present resolution not verified here \\
Comparison & Relate specified statements under fixed hypotheses & Expository tasks; some baseline implications already hold \\
Proposed task & Obtain missing data for a stated sufficient criterion or a translation model & A concrete target; novelty and present unresolvedness not claimed \\
\hline
\end{tabular}
\end{center}

\begin{remark}[Reading the questions]\label{rem:problem-list-status}
Each question fixes a foundation or class of objects and describes the data or conclusion sought. The historical subsection attributes its problems to a dated source. The comparison subsection distinguishes established consequences from additional representation data. The proposed-task subsection asks how to verify hypotheses that have been left explicit above; it does not remove those hypotheses from the proved results.
\end{remark}

\subsection{Questions explicitly posed in the literature}

\begin{question}[Historical question: left versus right almost periodicity]\label{q:left-right-ap}
Work in Bishop-style analysis with countable dependent choice and with the uniform finite-cover definition of translation total boundedness in Section~\ref{sec:known-results}. For an arbitrary topological group $G$ and bounded continuous $f:G\to\CC$, does left almost periodicity supply right almost periodicity in this same sense? More explicitly, can left finite-cover data be used, for every $\varepsilon>0$, to produce $t_1,\ldots,t_m\in G$ satisfying
\[
\forall t\in G\ \exists j\in\{1,\ldots,m\}\ \forall x\in G,
\qquad |f(xt)-f(xt_j)|<\varepsilon,
\]
without a classical selection of a possibly non-enumerable subfamily of a finite cover?
\end{question}

\noindent\emph{Attribution and scope.}
Spitters raises the constructive issue in the conclusion of \cite{SpittersAP}, referring to Loomis's Lemma~41B and the need for a finitely enumerable approximation family. This records the question in 2005, not a verified unresolvedness claim in 2026. For Abelian groups the two translation families agree. For a function supplied as a pullback along a dense homomorphism into a compact metric group, finite covers for both families can be obtained from uniform continuity of translations on that compact group. These special cases do not settle the stated general implication.

\begin{question}[Historical question: a universal Bohr object]\label{q:bohr}
As a precise point-free version of the Gelfand-based problem raised by Coquand and Spitters, fix intuitionistic set theory without dependent choice and take $G=(\RR,+)$, regarded as a localic group. Can one construct a compact regular localic group $B$ and a group morphism $\beta:G\to B$ with dense localic image such that, for every compact regular localic group $K$ and continuous group morphism $u:G\to K$, there is a unique continuous group morphism $\overline u:B\to K$ with
\[
u=\overline u\circ\beta?
\]
Any point-set recovery should be treated separately, with its own choice, normability, and separability assumptions specified.
\end{question}

\noindent\emph{Attribution and scope.}
The conclusion of \cite{CoquandSpittersGelfand} asks about constructing a Bohr compactification through constructive Gelfand theory. It explains an obstruction to a proposed algebraic route involving normability. The question above specifies one foundational version for investigation; it is not a claim that the authors used exactly this formulation or that no later construction exists. Its universal property is different from the single-function realization $H_f$ of \Cref{prop:compactification-principle}. Compact metric realizations of individual functions do not prove this universal assertion. One must also justify an algebraic construction rather than assume that all constructively almost periodic functions form an algebra.

\paragraph{Follow-up literature review (12 October 2026).}\label{sec:historical-followup}
The original left--right issue concerns finite enumerable translation covers, not the classical equivalence alone. A search of the source papers, indexed citing records, author bibliographies, and related constructive-analysis literature did not identify a later proof or counterexample under the precise hypotheses of Question~\ref{q:left-right-ap}. This is a report of the search outcome, not a theorem asserting that the question remains open.

There are substantial later developments in the background to Question~\ref{q:bohr}. Coquand and Spitters give fully constructive Gelfand duality for commutative $C^*$-algebras with compact completely regular localic spectra \cite{CoquandSpittersDuality}. Spitters relates locatedness to overtness \cite{SpittersLocated}; Henry extends duality to localic $C^*$-algebras and compact regular locales \cite{HenryLocalic}, and proves the continuous-norm/overtness criterion in the non-unital setting \cite[Theorem~5.2]{HenryNonunital}. These results supply representation tools; they do not by themselves construct the universal group morphism, extended group operations, or factorization property in Question~\ref{q:bohr}.

An important foundational caution is that compact localic groups need not be overt in general intuitionistic topos logic \cite[Section~4]{HenryFinite}. The overtness step in the 2005 discussion must therefore be checked in its intended setting before its normability obstruction is transported to a different foundation. Curi also distinguishes choice-free topos constructions of Stone--\v{C}ech compactification from their availability in predicative systems \cite{CuriStoneCech}. More recently, Manuell gives constructive uniform completions and proves that overt localic groups are complete in their two-sided uniformity \cite[Theorem~5.10]{ManuellUniform}. This is a completeness theorem, not the compact universal realization required here. The reviewed literature thus establishes related progress, but no same-setting resolution of Question~\ref{q:bohr} was identified. Point-set recovery, overtness, and the underlying set theory remain separate checks. The search is not a complete census of all publications; absence of an identified resolution does not establish unresolvedness.

\subsection{Comparison tasks and their established baseline}

The distinction between approximation and decomposition must not conceal a direct consequence of orthogonality. We record it before asking for further representation data.

\begin{lemma}[Spectral Hilbert-sum consequence]\label{lem:spectral-sum}
Under the hypotheses of \Cref{thm:CS-schematic}, write $E_\sigma=\operatorname{Ran}P_{\{\sigma\}}$. Let $\mathcal D$ consist of finite, explicitly supported families $v=(v_\sigma)$ with $v_\sigma\in E_\sigma$, equipped with
\[
\|v\|_{\mathcal D}^2=\sum_\sigma\|v_\sigma\|_2^2.
\]
Then summation extends to an isometric isomorphism from the Hilbert completion of $\mathcal D$ onto $L^2(G)$.
\end{lemma}
\begin{proof}
Mutual orthogonality gives $\|\sum_\sigma v_\sigma\|_2^2=\sum_\sigma\|v_\sigma\|_2^2$ for finite support, so summation is an isometry. Its image contains every $P_Ff$ and is dense by \Cref{thm:CS-schematic}. Completing this dense isometry identifies the two Hilbert completions. The already permitted countable dependent choice supplies successive finite approximations when a Cauchy representative is needed. No irreducibility assertion enters the argument.
\end{proof}

\begin{question}[Comparison task: splitting a finite spectral block]\label{q:irreducible-scope}
Retain the Bishop-style and Gelfand hypotheses of \Cref{thm:CS-schematic}. For a spectral block $E_\sigma$, what additional explicit finite-dimensional data suffices to construct an orthogonal decomposition
\[
E_\sigma=W_{\sigma,1}\oplus\cdots\oplus W_{\sigma,r}
\]
with left-translation invariant summands and a specified constructive irreducibility criterion? One possible positive criterion is that, for every $v\in W_{\sigma,j}$ with $\|v\|_2>0$, the closed span of its $G$-orbit equals $W_{\sigma,j}$. The desired output should include bases, the invariant projections, and any multiplicity or coefficient-space identifications being asserted.
\end{question}

\noindent\emph{What is already established.}
The blockwise Hilbert completion is supplied by \Cref{lem:spectral-sum}; it should not be listed as an unresolved consequence of strong approximation. The classical appendix supplies irreducible decompositions classically. A constructive comparison must check the relevant arguments in \cite{CoquandSpittersPW} and fix its meaning of irreducibility; it must not treat an entire central coefficient block, of classical dimension $d_\pi^2$, as a single irreducible representation. This is a comparison target, not a certified new open problem.

\begin{question}[Comparison task: hypotheses in point-set recovery]\label{q:locale-points}
For the commutative operator algebra $A_Z$ used in \Cref{thm:CS-schematic}, compare its point-free Gelfand spectrum with its point-set compact metric realization. Track separately the roles of normability, separability, and dependent choice in obtaining a positivity predicate, metric point data, and the norm-preserving representation. If one of these hypotheses is to be weakened, specify which one, what replaces it, and whether the conclusion remains point-free or is required to be point-set.
\end{question}

\noindent\emph{What is already established.}
The point-free and point-set representation results in \cite{CoquandSpittersGelfand} already address this comparison under stated assumptions. The task is first to match those assumptions to $A_Z$ and to the approximation operators of this article. It would be misleading to present the entire point-free versus point-set relation as an unstudied open question. A weakening of a hypothesis would be a different, separately formulated problem.

\begin{remark}[Directions of comparison]\label{rem:technical-frontier}
The historical questions concern left-right translation data and a universal compactification. The comparison tasks concern invariant splitting within finite spectral blocks and the hypotheses needed to recover point data. These are distinct questions; the spectral Hilbert-sum consequence is already settled under the current assumptions. None is an automatic consequence of resolving one of the others.
\end{remark}

\subsection{Concrete tasks proposed from the present formulation}

\begin{question}[Proposed task: data for a two-sided metric]\label{q:two-sided-data}
Let $G$ be a group presented on a countable set with decidable equality and given group operations. Let $f:G\to\CC$ be bounded and supplied with explicit left and right uniform finite-cover data. Determine additional conditions, if required, that produce located real values of
\[
\rho_f(a,b)=\sup_{s,t\in G}|f(sat)-f(sbt)|
\]
and, for every $\varepsilon>0$, a finite list $a_1,\ldots,a_n\in G$ such that every $a\in G$ has $\rho_f(a,a_i)<\varepsilon$ for some $i$. With the discrete topology, continuity into the new metric is automatic. The sought data are exactly the remaining metric inputs in case~(2) of \Cref{prop:compactification-principle}.
\end{question}

\noindent\emph{Boundary of the task.}
For a supplied finite group, both the supremum and the nets are obtained by finite enumeration; the $S_3$ example is already checked. The task concerns an infinite countable group with specified translation data. It is not a claim that the sufficient criterion is an unconditional general-group theorem, or that the proposed data-construction problem is known to be new.

\begin{question}[Proposed task: compatible joint realizations]\label{q:joint-realizations}
For the same presented group $G$, suppose two functions have supplied realizations from bi-invariant pseudometrics $\delta_1,\delta_2$ satisfying \Cref{prop:compactification-principle}. Give a constructive sufficient condition that produces finite nets with centers in $G$ for
\[
\delta_{12}(a,b)=\max\{\delta_1(a,b),\delta_2(a,b)\}.
\]
Bi-invariance, locatedness, and continuity of this maximum follow from the two inputs. What additional information on the joint image in $H_1\times H_2$ supplies the required finite nets, without assuming that every subset of a totally bounded space is constructively totally bounded with its own centers?
\end{question}

\noindent\emph{Boundary of the task.}
If the joint nets are supplied, the proof of \Cref{prop:compactification-principle} immediately gives a common compact metric realization through which both functions factor. The task is the construction of those nets. A common realization for two functions would still be different from the universal Bohr property in Question~\ref{q:bohr}.

\begin{question}[Proposed test: a translation approximation package]\label{q:minimal-package}
Let $Q$ be a compact metrizable topological loop supplied with a full-support probability measure $\mu$ preserved by all left and right translations. On $\mathcal H=L^2(Q,\mu)$ let
\[
(U_a h)(x)=h(L_a^{-1}x),\qquad
(V_a h)(x)=h(R_a^{-1}x),
\]
where $L_a x=ax$ and $R_a x=xa$ are homeomorphisms. What explicit additional structure supplies finite-rank orthogonal projections $P_F$ with continuous ranges, commuting with every $U_a,V_a$, and satisfying $\|P_Fh-h\|_2\to0$ for every $h\in\mathcal H$? The requirements are now operator invariance and a specified norm, rather than an undefined ``minimal Peter--Weyl package.''
\end{question}

\begin{question}[Proposed route: compactness of the translation operator group]\label{q:ambient-category}
For the measured loops in Question~\ref{q:minimal-package}, which additional conditions make the group generated by the unitary operators $U_a,V_a$ have a compact closure in the strong operator topology, with continuous action on $\mathcal H$? First test finite loops, where the translation group is a finite permutation group, and then a specified infinite class. A constructive version must also supply the metric and analytic data needed to apply a compact-group theorem, rather than use classical compactness alone.
\end{question}

\begin{remark}[Status of the proposed nonassociative tasks]\label{rem:programmatic-aim}
The last two questions fix a measure, a Hilbert space, translation operators, and an approximation target. Section~\ref{sec:loops} gives a sufficient classical compact transformation-group criterion, while the broader constructive targets remain proposed routes for investigation. A finite loop is only a baseline check: the identity projection on its finite-dimensional function space already gives approximation. The infinite case requires additional structure, and any novelty claim must be checked against existing translation-group and nonassociative harmonic-analysis literature. Compact quantum groups are not included as an undifferentiated open direction here; their existing theories would need a separate survey.
\end{remark}

\section{Further Outlook: Loops and Quasigroups}\label{sec:loops}

The measure-theoretic frameworks in \cite{InoueQuasi,InoueLoop} motivate an investigation of translation-generated approximation spaces. They discuss quasi-invariant measures and modular data; such data do not by themselves supply compactness of a translation group or a Peter--Weyl approximation theorem. We specify an algebraic obstruction, a classical sufficient transformation-group criterion, and the additional checks needed for a constructive extension. The classical criterion below is an application of compact-group theory, not a claim of a new general nonassociative Peter--Weyl theorem.

\subsection{An obstruction to ordinary multiplicative representations}

\begin{lemma}[Point separation by multiplicative maps forces associativity]\label{lem:representation-obstruction}
Let $Q$ be a nonempty quasigroup. Suppose a family of maps $\rho_i:Q\to K_i$, with each $K_i$ a group, satisfies $\rho_i(xy)=\rho_i(x)\rho_i(y)$ and jointly separates points of $Q$. Then $Q$ is associative and hence a group.
\end{lemma}
\begin{proof}
For every $i$ and $x,y,z\in Q$,
\[
\rho_i((xy)z)=\rho_i(x)\rho_i(y)\rho_i(z)=\rho_i(x(yz)).
\]
Point separation gives $(xy)z=x(yz)$. To recall why an associative quasigroup is a group, choose $a$ and solve $ae=a$. Associativity and left cancellation give $ex=x$ for every $x$. Then $(xe)y=x(ey)=xy$, and right cancellation gives $xe=x$. The left and right inverses supplied by division coincide by associativity.
\end{proof}

Consequently, ordinary product-preserving unitary representations cannot separate the points of a genuinely nonassociative loop. A translation-generated representation uses a different group: its elements are transformations of $Q$, not elements of $Q$ with their original product. This distinction is necessary even before questions of compactness or constructivity arise.

For a loop with identity $e$, put $L_a x=ax$, $R_a x=xa$, and let
\[
M(Q)=\langle L_a,R_a:a\in Q\rangle\subseteq\operatorname{Homeo}(Q).
\]
The associativity defect can be recorded by the transformation
\[
A_{a,b}=L_{ab}^{-1}L_aL_b,\qquad
A_{a,b}(e)=e,\qquad L_aL_b=L_{ab}A_{a,b}.
\]
Here inverses are inverses of translations, defined by division; no associative product on $Q$ is being assumed. In general $A_{a,b}$ is not the identity.

\subsection{Translation operators and measure correction}

First suppose a full-support probability measure $\mu$ is preserved by the left and right translations. Their compositions preserve $\mu$ as well. For $T\in M(Q)$ define
\[
(\mathscr U_T h)(x)=h(T^{-1}x)
\quad\text{on }\mathcal H=L^2(Q,\mu).
\]
Then $\mathscr U_T$ is unitary and $\mathscr U_T\mathscr U_S=\mathscr U_{TS}$. The notation of Question~\ref{q:minimal-package} is $U_a=\mathscr U_{L_a}$ and $V_a=\mathscr U_{R_a}$. In a loop the defect relation becomes
\[
U_aU_b=U_{ab}\mathscr U_{A_{a,b}}.
\]
Unitarity therefore does not turn $a\mapsto U_a$ into a multiplicative representation of the loop.

For a quasi-invariant measure, the correction must have an explicit convention. Suppose $T$ is an invertible nonsingular transformation and the positive Radon--Nikodym derivative
\[
j_T(x)=\frac{d(T_*\mu)}{d\mu}(x),\qquad
(T_*\mu)(E)=\mu(T^{-1}E)
\]
is supplied. The corrected unitary pullback is
\[
(\mathscr W_T h)(x)=j_T(x)^{1/2}h(T^{-1}x).
\]
Indeed, integrating against $T_*\mu$ gives $\|\mathscr W_T h\|_2=\|h\|_2$. The chain rule, with transformations composed as $(TS)x=T(Sx)$, is
\[
j_{TS}(x)=j_T(x)j_S(T^{-1}x)
\quad\text{almost everywhere},
\]
and yields $\mathscr W_T\mathscr W_S=\mathscr W_{TS}$. Thus the corrected operators still retain the transformation $A_{a,b}$ when translations are composed. Strong continuity of these weighted operators, compactness of their generated group, and preservation of continuous functions are additional requirements. The square-root correction alone proves none of them. In a constructive treatment the derivative and its needed analytic properties must be supplied or proved, not imported from a classical Radon--Nikodym theorem.

\subsection{A sufficient classical compact transformation-group criterion}

For a compact metrizable $Q$, fix a compatible metric $d$ and give $\operatorname{Homeo}(Q)$ the topology of uniform convergence of maps and their inverses, induced by
\[
d_\pm(S,T)=\sup_{x\in Q}d(Sx,Tx)
+\sup_{x\in Q}d(S^{-1}x,T^{-1}x).
\]
The closure $\Gamma=\overline{M(Q)}$ is taken \emph{in this topological group}. A closure in all continuous self-maps would not ensure invertibility of limits. A strongly compact unitary-operator closure alone would also not supply the continuous point action required below.

\begin{proposition}[Classical approximation from a compact transformation group]\label{prop:classical-translation-PW}
Let $Q$ be a nonempty compact metrizable topological quasigroup, with multiplication and both division operations continuous. Let $\mu$ be a full-support Radon probability measure preserved by every left and right translation. Assume that $\Gamma=\overline{M(Q)}$ is compact for the topology just specified. Then there is a directed family of finite-rank orthogonal projections $\Pi_F$ on $L^2(Q,\mu)$, with continuous ranges, such that
\[
\Pi_F\mathscr U_T=\mathscr U_T\Pi_F\quad(T\in M(Q)),
\qquad \|\Pi_Fh-h\|_2\longrightarrow0.
\]
These are classical conclusions. They do not assert uniform approximation or an unconditional constructive result for all compact quasigroups.
\end{proposition}

\begin{proof}
Uniform limits of $\mu$-preserving homeomorphisms preserve the integrals of continuous functions, so $\Gamma$ preserves $\mu$. Its action on $Q$ is continuous and transitive: for any fixed $o\in Q$ and $x\in Q$, division supplies $a$ with $ao=x$, and hence $L_a(o)=x$.

Let $m$ be normalized Haar measure on the compact group $\Gamma$. For $h\in C(Q)$, the average
\[
\int_\Gamma h(\gamma x)\,dm(\gamma)
\]
is independent of $x$, by transitivity and right invariance of $m$. Integrating this constant against the $\Gamma$-invariant measure $\mu$ shows that
\[
\int_Q h\,d\mu=\int_\Gamma h(\gamma o)\,dm(\gamma).
\]
Thus the orbit map $q_o(\gamma)=\gamma o$ pulls back functions by an isometry
\[
J:L^2(Q,\mu)\longrightarrow L^2(\Gamma,m),\qquad Jh=h\circ q_o.
\]
The action $\gamma\mapsto\mathscr U_\gamma$ is strongly continuous: this follows first in the supremum norm for continuous functions and then by $L^2$ density and unitarity. Moreover $J\mathscr U_\gamma=L_\gamma J$.

Use the classical central idempotent kernels $p_\pi=d_\pi c_\pi$ on $\Gamma$, with the convention of Section~\ref{sec:known-results}, and define
\[
\Pi_\pi h=\int_\Gamma p_\pi(\gamma)\mathscr U_\gamma h\,dm(\gamma).
\]
Centrality and convolution idempotence make these mutually orthogonal projections commuting with the action. They intertwine under $J$ with the finite-rank regular projections $P_\pi^\Gamma$:
\[
J\Pi_\pi=P_\pi^\Gamma J.
\]
By \Cref{thm:classical-PW}, the image of $P_\pi^\Gamma$ has dimension $d_\pi^2$, so $\operatorname{rank}\Pi_\pi\le d_\pi^2$. The integral formula maps $C(Q)$ into $C(Q)$ by continuity of the compact action. Since $\Pi_\pi C(Q)$ is a dense linear subspace of the finite-dimensional range of $\Pi_\pi$, it is that whole range. Thus the range consists of continuous representatives.

For a finite set of spectral indices put $\Pi_F=\sum_{\pi\in F}\Pi_\pi$. The regular Peter--Weyl approximation in the appendix gives
\[
\|h-\Pi_Fh\|_{L^2(Q)}
=\|Jh-P_F^\Gamma Jh\|_{L^2(\Gamma)}\longrightarrow0.
\]
The compact-group method is classical; see also the representation and projection background in \cite{MurnaghanNotes}.
\end{proof}

\begin{corollary}[Classical common-isometry-metric criterion]\label{cor:translation-isometries}
Suppose a nonempty compact metrizable topological quasigroup $Q$ has a compatible metric satisfying
\[
d(L_a x,L_a y)=d(x,y)=d(R_a x,R_a y)
\quad(a,x,y\in Q).
\]
Then classically it admits a full-support translation-invariant Radon probability measure, and the approximation conclusion of \Cref{prop:classical-translation-PW} holds for that measure.
\end{corollary}
\begin{proof}
Every element of $M(Q)$ is an isometry for the common metric. The isometry group of a compact metric space is compact: Arzel\`a--Ascoli applies to maps and inverses, and a simultaneous uniform limit of inverse pairs is again an inverse pair. Thus $\Gamma=\overline{M(Q)}$ is compact. Push normalized Haar measure on $\Gamma$ forward under $q_o$. Transitivity makes $q_o$ surjective, so every nonempty open subset of $Q$ has a nonempty open inverse image and positive measure. The resulting probability measure is $\Gamma$-invariant, and the preceding proposition applies.
\end{proof}

Compactness of $Q$ alone is not this criterion. Even when the individual families $\{L_a\}$ and $\{R_a\}$ are equicontinuous, one needs control of compositions of arbitrary word length in $M(Q)$. A common invariant metric supplies such control. This identifies an actual structural condition to test in an infinite class, rather than assume that the multiplication group is compact merely because $Q$ is compact.

\subsection{A fully explicit finite nonassociative example}

On $Q=\{0,1,2,3,4\}$ define the multiplication table
\[
\begin{array}{c|ccccc}
\cdot&0&1&2&3&4\\\hline
0&0&1&2&3&4\\
1&1&0&3&4&2\\
2&2&3&4&0&1\\
3&3&4&1&2&0\\
4&4&2&0&1&3
\end{array}
\]
Each row and column is a permutation and $0$ is an identity, so this is a loop. It is nonassociative, since
\[
(1\cdot1)\cdot2=2,\qquad 1\cdot(1\cdot2)=4.
\]
All translations preserve the discrete metric and uniform measure. In cycle notation,
\[
L_1=(0\ 1)(2\ 3\ 4),\qquad L_2=(0\ 2\ 4\ 1\ 3).
\]
The power $L_1^3=(0\ 1)$ is a transposition. Its conjugates by powers of $L_2$ are
\[
(0\ 1),\quad(2\ 3),\quad(4\ 0),\quad(1\ 2),\quad(3\ 4).
\]
They include all adjacent transpositions, so $M(Q)=S_5$. Its permutation representation on the five-dimensional function space is a representation of $S_5$, not a multiplicative representation of this loop.

Writing functions as column vectors, the projections
\[
P_{\mathrm{const}}=\tfrac15\mathbf1\mathbf1^*,\qquad
P_{\mathrm{zero}}=I-P_{\mathrm{const}}
\]
have ranks $1$ and $4$, commute with every translation permutation, and sum to the identity. The defect $A_{1,1}=L_1^2=(2\ 4\ 3)$ fixes the identity but is nontrivial. Thus $U_1^2=\mathscr U_{A_{1,1}}\ne I$, whereas a multiplicative unitary map would have $\rho(1)^2=\rho(0)=I$. This verifies the distinction between the loop product, its translation group, and its invariant approximation spaces. The finite case is a mechanism check, not evidence that the infinite compactness hypotheses hold automatically.

\subsection{What remains to be supplied constructively}

The compact transformation-group argument just given uses classical compactness and measure theory. A constructive adaptation must specify:
\begin{enumerate}
\item a compatible metric and finite approximation data for the closure of the generated transformations, including inverse maps and arbitrary compositions;
\item a compact group realization with the Haar, normability, separability, and Gelfand data needed for the regular constructive theorem;
\item a justified orbit pullback and finite-rank projections on the target function space, with continuous ranges and a stated error norm;
\item in the quasi-invariant case, the derivative data, strong continuity, and regularity needed for corrected pullbacks.
\end{enumerate}
Neither the classical Arzel\`a--Ascoli step nor the classical recovery of continuous finite-dimensional ranges is silently asserted in a constructive foundation. The regular compact-group theorem of Section~\ref{sec:known-results} is a useful input, but applying it to a translation-generated action on a different space needs these additional arguments.

Finally, a choice of the reference point $o$ affects the orbit pullback $J$, while the averaging formula for $\Pi_\pi$ is independent of $o$. An isotopy of the quasigroup may change the multiplication and the distinguished left and right translations. Comparing approximants across an isotopy therefore needs an explicit intertwining map, not just an arbitrary change of reference point. This gives a concrete later comparison task without claiming that such functoriality has already been established.

\section{A Suggested Program for Future Work}\label{sec:program}

The status assessment separates the established compact-group theorem from comparisons, historical problems, and proposed extensions. The following program specifies work in these latter categories:
\begin{enumerate}
\item Examine which normability, separability, and choice assumptions in the compact-group formulation can be verified or weakened in a specified class of examples.
\item Address the two-sided metric data and compatible joint realizations specified in Questions~\ref{q:two-sided-data} and~\ref{q:joint-realizations}. These tasks identify exactly which completion inputs still need to be constructed.
\item Compare finite-block splitting and point-set recovery under fixed foundations, as in Questions~\ref{q:irreducible-scope} and~\ref{q:locale-points}. Keep these comparisons separate from rechecking the current status of the historical Questions~\ref{q:left-right-ap} and~\ref{q:bohr}.
\item In a specified infinite class, verify the compact transformation-group or common-isometry-metric hypotheses of \Cref{prop:classical-translation-PW,cor:translation-isometries}. Then address the constructive approximation target of Questions~\ref{q:minimal-package} and~\ref{q:ambient-category}, or the extra regularity required by quasi-invariant corrected pullbacks.
\end{enumerate}
These are proposed tasks. The results already stated in this article should not be deferred to a hypothetical later exposition, and the tasks above are not claimed to be a list of verified new open problems.

\section{Conclusion}\label{sec:conclusion}

\emph{What is known?} Constructive Peter--Weyl theory has an established compact-group approximation core in the Bishop--Coquand--Spitters line. Under the stated analytic and Gelfand hypotheses, finite-rank central convolution projections converge strongly in the Haar $L^2$ norm. The kernel normalization is explicit, and orthogonality gives the spectral Hilbert completion of \Cref{lem:spectral-sum}. These are proved statements, not merely a program for obtaining a future constructive theorem.

\emph{What remains open, or requires further clarification?} The survey distinguishes three kinds of questions. First, comparisons of irreducible representation data, uniform approximation, and point-set recovery require their own formulations and a match with existing results; absence of a proof here does not establish that they are unknown. Second, the historically posed left--right almost-periodicity and universal Bohr problems have a dated follow-up review: related progress is identified, but a resolution in the specified settings was not found. This search outcome does not certify current unresolvedness. Third, the two-sided and joint realization tasks and the constructive nonassociative targets are explicit proposed investigations. Their status is not that of new open problems established by this article.

The resulting answer to ``is constructive Peter--Weyl theory complete?'' is affirmative for the precise compact-group approximation core. A claim about completion of the whole classical package across foundations requires additional specification and justification. This overall assessment follows from the distinction between established results and the statements that remain to be examined.

The metric realization criterion and its $S_3$ example clarify how the established approximation theorem can be applied to almost periodic data. The loop and quasigroup results describe a conditional classical route and the further constructive requirements for extending it. These questions and examples form part of the wider picture, while remaining distinct from the proof status of the compact-group core. The classical appendix supplies both a detailed proof and a comparison at each stage, so that the reader can identify the change from classical existence arguments to specified constructive approximation data.

\appendix

\section{The classical Peter--Weyl theorem and a standard Haar-measure proof}\label{app:classical-pw}

This appendix serves two purposes: to give a detailed classical proof and to show where its reasoning must be supplemented or reorganized constructively. We follow the passage from Haar integration to finite-dimensional invariant spaces, then to density, separation, and irreducible coefficients. See \cite{HewittRoss,Loomis,DuistermaatKolk} for classical background. The theorem and its proof below are classical; the accompanying comparison paragraphs distinguish transferable calculations from steps requiring separately supplied constructive data.

The comparison is not based on declaring Haar integration, spectral theory, or Stone--Weierstrass inherently nonconstructive. Constructive versions and substitutes exist under appropriate assumptions \cite{BishopBridges,CoquandSpittersPW,BanaschewskiMulveySW}. The task is to identify exactly what the present classical argument obtains automatically and what the constructive formulation must specify. The main text works with compact metric data and a normable Gelfand input, whereas this appendix starts with an arbitrary compact Hausdorff group. That change of hypotheses already matters before any operator is diagonalized.

The proof uses right convolution with possibly noncentral kernels. The main-text approximation uses left convolution with central idempotent kernels. Both routes obtain translation-invariant finite-dimensional spaces, but the commuting calculation and the mechanism selecting those spaces must be kept distinct. The comparisons at Steps~2--4 make this distinction explicit.

Throughout, $G$ is a compact Hausdorff group and $\mu$ is its normalized Haar probability measure. For compact groups, $\mu$ is left invariant, right invariant, and invariant under inversion. It has full support: every nonempty open set has positive measure. Indeed, finitely many translates of any nonempty open set cover $G$. We use inner products linear in the first variable, and regard $L^2(G)$ as the completion of $C(G)$ for the Haar inner product. Full support identifies each continuous representative uniquely.

\subsection{Statement of the theorem}

For a finite-dimensional continuous unitary representation $\pi:G\to U(V)$ and $v,w\in V$, the function
\[
\varphi_{v,w}(x)=\langle\pi(x)v,w\rangle
\]
is a \emph{matrix coefficient}. Let $\mathcal A(G)$ be the complex linear span of all such functions. Choose one representative $\pi$ of each equivalence class of irreducible unitary representations, and write $d_\pi=\dim H_\pi$.

\begin{theorem}[Classical Peter--Weyl theorem]\label{thm:classical-PW}
Let $G$ be a compact Hausdorff group.
\begin{enumerate}
\item Every finite-dimensional continuous complex representation of $G$ is equivalent to a unitary representation.
\item $\mathcal A(G)$ is a point-separating $*$-subalgebra of $C(G)$ containing the constants, and is uniformly dense in $C(G)$.
\item For the chosen irreducible representatives and an orthonormal basis of each $H_\pi$, their matrix coefficients, multiplied by $\sqrt{d_\pi}$, form an orthonormal basis of $L^2(G)$.
\item If $\mathcal C_\pi$ denotes the coefficient space of $\pi$, then
\[
L^2(G)=\widehat{\bigoplus}_{\pi\in\widehat G}\mathcal C_\pi,
\qquad
\mathcal C_\pi\cong H_\pi\otimes H_\pi^*,
\qquad \dim\mathcal C_\pi=d_\pi^2.
\]
More precisely, the unitary identification sends $v\otimes\ell_w$ to $\sqrt{d_\pi}\,\varphi_{v,w}$, where $\ell_w(u)=\langle u,w\rangle$. For $L_gf(x)=f(g^{-1}x)$, its equivariance is
\[
L_g\big|_{\mathcal C_\pi}\cong I_{H_\pi}\otimes\pi^*(g),
\qquad \pi^*(g)\ell=\ell\circ\pi(g^{-1}).
\]
Thus this summand contains $d_\pi$ copies of the contragredient representation $\pi^*$, with the coefficient convention fixed above.
\end{enumerate}
\end{theorem}

\subsection{Step 1: Haar averaging}

\begin{lemma}[Haar averaging]\label{lem:haar-average}
Every finite-dimensional continuous representation $\rho:G\to\mathrm{GL}(V)$ preserves a Hermitian inner product.
\end{lemma}
\begin{proof}
Choose an inner product $\langle\cdot,\cdot\rangle_0$ and put
\[
\langle v,w\rangle=\int_G\langle\rho(x)v,\rho(x)w\rangle_0\,d\mu(x).
\]
The integrand is continuous. For $v\ne0$, its value on the diagonal is strictly positive at $e$, and hence on a neighborhood of positive Haar measure. Thus the averaged form is positive definite. There is also a useful quantitative estimate. Since $x\mapsto\rho(x)^{-1}$ is continuous on a compact space, choose $M\ge1$ bounding its operator norm. Then
\[
\|\rho(x)v\|_0\ge M^{-1}\|v\|_0,
\qquad \langle v,v\rangle\ge M^{-2}\|v\|_0^2.
\]
This estimate makes the strict positivity of the averaged form explicit. Right invariance, with $y=xg$, gives
\[
\langle\rho(g)v,\rho(g)w\rangle
=\int_G\langle\rho(xg)v,\rho(xg)w\rangle_0\,d\mu(x)
=\langle v,w\rangle.
\]
\end{proof}

\paragraph{Constructive comparison: the integral and its positivity.}
Once invariant integration and explicit finite-dimensional representation data are supplied, the averaging formula and change of variables remain the same. The quantitative lower bound above shows what is required for a constructively positive definite inner product: a positive bound on inverses, not just a negative assertion that no nonzero vector has zero length. In a compact metric setting, finite matrices, uniform continuity, and finite nets can supply such a bound. The main text specifies Haar integration as input; the appendix does not prove its existence constructively for every compact Hausdorff group. Haar averaging itself is therefore not the point at which the proof must fail constructively. The distinction is between having the integral and its positive estimates as data and obtaining them from unrestricted classical existence results.

\subsection{Step 2: right convolution and left translation}

We retain the convolution convention of the main text,
\[
(f*k)(x)=\int_G f(y)k(y^{-1}x)\,d\mu(y),
\]
but use \emph{right convolution} as our approximation operator:
\[
S_kf=f*k,\qquad (L_gf)(x)=f(g^{-1}x).
\]
Right convolution commutes with left translation for arbitrary continuous kernels. Left convolution $f\mapsto k*f$ need not do so on a noncommutative group; it does when $k$ is central. No centrality assumption is needed below.

\begin{lemma}\label{lem:convolution-compact}
For $k\in C(G)$, $S_k:L^2(G)\to L^2(G)$ is compact, and $S_k$ maps $L^2(G)$ continuously into $C(G)$. If $k(x^{-1})=\overline{k(x)}$, then $S_k$ is self-adjoint.
\end{lemma}
\begin{proof}
The continuous kernel is $K(x,y)=k(y^{-1}x)$. Cauchy--Schwarz and Haar invariance give
\[
\|S_kf\|_\infty\le\|k\|_2\|f\|_2,
\quad
|S_kf(x)-S_kf(x')|
\le\|f\|_2\|k(\,\cdot\,^{-1}x)-k(\,\cdot\,^{-1}x')\|_2.
\]
Continuity of $K$ on the compact product implies that the last norm tends to zero as $x'\to x$; hence $S_kf$ is continuous.

Finite sums $\sum_{j=1}^n a_j(x)b_j(y)$, with $a_j,b_j\in C(G)$, uniformly approximate $K$ by the classical Stone--Weierstrass theorem on $G\times G$. Their integral operators have finite rank, and a uniform kernel error $\delta$ gives an operator-norm error at most $\delta$, since $\mu(G)=1$. Therefore $S_k$ is compact. This argument also applies when $L^2(G)$ is nonseparable.

Finally the adjoint kernel is
\[
K^*(x,y)=\overline{k(x^{-1}y)}.
\]
The stated symmetry of $k$ makes this equal to $K(x,y)$, proving self-adjointness.
\end{proof}

\begin{lemma}\label{lem:intertwining}
For every $g\in G$ and $k\in C(G)$, $L_gS_k=S_kL_g$. Consequently, every eigenspace of $S_k$ is invariant under left translation.
\end{lemma}
\begin{proof}
For $f\in L^2(G)$, left invariance and the substitution $y=gz$ yield
\[
\begin{split}
(S_kL_gf)(x)
&=\int_G f(g^{-1}y)k(y^{-1}x)\,d\mu(y)\\
&=\int_G f(z)k(z^{-1}g^{-1}x)\,d\mu(z)
=(L_gS_kf)(x).
\end{split}
\]
The order of the factors is unchanged throughout this calculation.
\end{proof}

\paragraph{Constructive comparison: finite-rank kernel approximation.}
The intertwining identity and the Cauchy--Schwarz estimates are direct calculations. They do not require choosing irreducible representations or eigenvectors. The classical compactness proof above uses Stone--Weierstrass on $G\times G$. This is not circular: that theorem approximates the kernel by products of arbitrary continuous functions, before density of representation coefficients has been proved.

In the compact metric setting, the finite-rank approximation can instead be made explicit. Suppose $K$ is supplied with uniform continuity data. For $\eta>0$, choose $r>0$ such that
\[
d(y,y')<r\quad\Longrightarrow\quad
|K(x,y)-K(x,y')|<\eta\quad\text{for all }x.
\]
Take a finite $r/2$-net $y_1,\ldots,y_N$ in $G$ and set
\[
w_j(y)=\max(0,r-d(y,y_j)),\qquad
b_j(y)=\frac{w_j(y)}{\sum_{i=1}^Nw_i(y)}.
\]
The denominator is at least $r/2$, the weights sum to one, and $b_j(y)>0$ implies $d(y,y_j)<r$. Hence
\[
K_\eta(x,y)=\sum_{j=1}^NK(x,y_j)b_j(y),
\qquad \|K-K_\eta\|_\infty\le\eta.
\]
Its integral operator is
\[
S_\eta f(x)=\sum_{j=1}^NK(x,y_j)\int_Gb_j(y)f(y)\,d\mu(y),
\]
with rank at most $N$. The probability normalization and Cauchy--Schwarz give $\|S_k-S_\eta\|_{2\to2}\le\eta$. This explains why compact convolution is compatible with constructive finite approximation under explicit metric and integration data. It does not yet supply a spectral decomposition. Also, right convolution commutes with $L_g$ for all kernels; the main-text left-convolution projections commute because their kernels are central.

\subsection{Step 3: finite-dimensional invariant spaces}

\begin{lemma}\label{lem:eigenspaces}
Suppose $k(x^{-1})=\overline{k(x)}$. Every eigenspace of $S_k$ corresponding to a nonzero eigenvalue is finite-dimensional, left-translation invariant, and consists of continuous representatives. The restricted left regular action is a continuous unitary representation.
\end{lemma}
\begin{proof}
Compact self-adjointness makes every eigenspace $E_\lambda$ with $\lambda\ne0$ finite-dimensional. To see the finite-dimensionality directly, an infinite-dimensional eigenspace would classically contain an orthonormal sequence $(v_n)$. Its images satisfy
\[
\|S_kv_n-S_kv_m\|_2=|\lambda|\sqrt2\quad(n\ne m),
\]
contradicting total boundedness of the image of the unit ball under a compact operator. The zero eigenvalue is excluded: its eigenspace need not be finite-dimensional. Invariance follows from \Cref{lem:intertwining}. If $S_kf=\lambda f$, then $\lambda^{-1}S_kf$ is its continuous representative by \Cref{lem:convolution-compact}.

Left translation preserves the Haar inner product. For every continuous $v$, compactness and continuity give $\|L_gv-L_{g_0}v\|_\infty\to0$ as $g\to g_0$. Applying this to a finite basis of $E_\lambda$ proves continuity of the restricted representation in operator norm.
\end{proof}

\paragraph{Constructive comparison: obtaining the spaces.}
The classical argument supplies eigenspaces and then their dimensions and bases. Constructively, the existence and identification of the relevant eigenvalues and finite-dimensional spaces need justification; one cannot replace this by an instruction to diagonalize an arbitrary compact operator. A statement that a spectral value is merely nonzero must also be distinguished from a supplied positive bound on its absolute value when division or effective separation is needed. Once a continuous finite basis is available, invariance, Haar unitarity, and continuity of the restricted action are finite calculations.

The main text organizes this step differently. Its normable central Gelfand input supplies mutually orthogonal convolution idempotents $p_\sigma$, and hence projections with finite-dimensional continuous ranges. It uses the constructive approximation result of \cite{CoquandSpittersPW}, rather than import the appendix's global eigenspace theorem. This is a change in the proof mechanism, not an assertion that constructive spectral theory does not exist. Finite-rank approximation of an operator, the existence of a particular eigenspace, and an explicit basis for that space are distinct kinds of data.

\subsection{Step 4: density in \texorpdfstring{$L^2(G)$}{L2(G)}}

Choose a net $(u_U)$ indexed by symmetric neighborhoods $U$ of $e$, ordered by shrinking, such that
\[
u_U\in C(G),\quad u_U\ge0,\quad
\int_Gu_U\,d\mu=1,\quad
\operatorname{supp}u_U\subseteq U,\quad u_U(x^{-1})=u_U(x).
\]
To construct it, choose a nonnegative continuous function $v$ supported in $U$ with $v(e)>0$, replace it by $v(x)+v(x^{-1})$, and normalize its positive integral. Such a $v$ exists by regularity and normality of a compact Hausdorff space.
Inversion invariance gives
\[
(S_{u_U}f)(x)=\int_G u_U(z)f(xz^{-1})\,d\mu(z).
\]
For continuous $f$, uniform continuity on the compact group therefore implies $\|S_{u_U}f-f\|_\infty\to0$. Moreover, since $u_U\,d\mu$ is a probability measure, Jensen's inequality gives
\[
|S_{u_U}f(x)|^2\le\int_Gu_U(z)|f(xz^{-1})|^2\,d\mu(z).
\]
Integration in $x$, interchange of integrals, and right invariance yield $\|S_{u_U}f\|_2\le\|f\|_2$. To extend convergence from continuous functions, choose $h\in C(G)$ close to $f\in L^2(G)$ and use
\[
\|S_{u_U}f-f\|_2
\le 2\|f-h\|_2+\|S_{u_U}h-h\|_2.
\]
First make the first term small, then shrink $U$ to control the second. Thus these contractions converge strongly on all of $L^2(G)$; the conclusion is not operator-norm convergence to the identity.

\begin{proposition}\label{prop:L2-density}
Let $\mathcal H_{\mathrm{fin}}$ be the span of all eigenspaces corresponding to nonzero eigenvalues of the self-adjoint operators $S_k$, for continuous symmetric kernels $k$. Then $\mathcal H_{\mathrm{fin}}$ is dense in $L^2(G)$.
\end{proposition}
\begin{proof}
For each $U$, the compact self-adjoint spectral theorem gives
\[
\overline{\operatorname{Ran}S_{u_U}}
=(\ker S_{u_U})^\perp
=\overline{\operatorname{span}\{E_\lambda(S_{u_U}):\lambda\ne0\}}.
\]
Each space on the right belongs to $\mathcal H_{\mathrm{fin}}$, by \Cref{lem:eigenspaces}. Thus $S_{u_U}f$ is in its closure. Passing to the $L^2$ limit proves the assertion. No countability or finite-dimensionality of the zero eigenspace is required.
\end{proof}

\paragraph{Constructive comparison: where spectral density enters.}
There are two separate ingredients in this stage. The approximate identity gives strong convergence of $S_{u_U}$, while the classical compact self-adjoint spectral theorem places each $S_{u_U}f$ in the closure of finite-dimensional eigenspaces. Constructively, supplying kernels supported near $e$ needs explicit neighborhood or metric data and a positive integral for normalization. The general compact Hausdorff existence of such functions is not the compact metric finite-net data assumed in the main text. Once kernels with the stated properties are available, the displayed contraction and two-error estimate remain valid.

The spectral range identity is the further classical input. To see the density argument in finite terms, given $f$ and $\varepsilon>0$, first choose $U$ with $\|f-S_{u_U}f\|_2<\varepsilon/2$, then use that identity to choose a finite sum $h$ of nonzero-eigenvalue vectors with $\|S_{u_U}f-h\|_2<\varepsilon/2$. Their sum gives $\|f-h\|_2<\varepsilon$. The classical argument guarantees these choices; a constructive version must supply their finite approximation content.

In \Cref{thm:CS-schematic}, the constructive finite-tail estimate supplies that content using the central projections $P_F$. It also yields the blockwise Hilbert completion of \Cref{lem:spectral-sum}. Neither route justifies inferring uniform convergence of those orthogonal partial sums from their $L^2$ convergence. Uniform approximation will be proved classically at the next stage by a different argument.

\subsection{Step 5: coefficients, point separation, and uniform density}

\begin{lemma}[Invariant functions are coefficients]\label{lem:invariant-coefficients}
Every function in a finite-dimensional left-translation invariant subspace $V\subset C(G)$ is a matrix coefficient of a finite-dimensional continuous unitary representation, possibly a conjugate representation. In particular, $\mathcal H_{\mathrm{fin}}\subseteq\mathcal A(G)$.
\end{lemma}
\begin{proof}
Use the Haar inner product on $V$ and put $\pi(g)=L_g|_V$. The same continuity argument as in \Cref{lem:eigenspaces} makes $\pi$ continuous and unitary. Evaluation $\ell(v)=v(e)$ is a linear functional on a finite-dimensional space. With an orthonormal basis $v_1,\ldots,v_n$, the vector
\[
w=\sum_{j=1}^n\overline{\ell(v_j)}\,v_j
\]
satisfies $\ell(v)=\langle v,w\rangle$, in our convention that the inner product is linear in the first variable. Hence
\[
v(x)=\ell(\pi(x^{-1})v)
=\langle v,\pi(x)w\rangle
=\overline{\langle\pi(x)w,v\rangle}.
\]
This is a coefficient of the conjugate unitary representation. The final assertion follows from \Cref{lem:eigenspaces}.
\end{proof}

\begin{proposition}\label{prop:separate-points}
Matrix coefficients of finite-dimensional continuous unitary representations separate points of $G$.
\end{proposition}
\begin{proof}
Suppose distinct $a,b\in G$ have equal values for every coefficient. Equality for all vectors in each representation gives $\pi(a)=\pi(b)$, so $g=a^{-1}b\ne e$ satisfies $\pi(g)=I$ in every such representation. All coefficients, and therefore all elements of $\mathcal H_{\mathrm{fin}}$, are fixed by $L_g$. By \Cref{prop:L2-density} and boundedness, $L_g=I$ on $L^2(G)$.

Choose $h\in C(G)$ with $h(e)\ne h(g^{-1})$. Then $L_gh-h$ is nonzero at $e$, and its absolute value is bounded below by a positive constant on some nonempty open set. Full support of Haar measure implies $\|L_gh-h\|_2>0$, a contradiction.
\end{proof}

\begin{proposition}\label{prop:uniform-density}
$\mathcal A(G)$ is a $*$-subalgebra containing constants and is uniformly dense in $C(G)$.
\end{proposition}
\begin{proof}
Products of coefficients are coefficients of tensor-product representations, sums belong to the span, and conjugates are coefficients of conjugate representations. The trivial representation supplies the constants. By \Cref{prop:separate-points} and the classical complex Stone--Weierstrass theorem, $\mathcal A(G)$ is uniformly dense.
\end{proof}

\paragraph{Constructive comparison: separation needs a witness.}
Once a finite-dimensional invariant space and a basis are supplied, the evaluation vector above and the coefficient formula are explicit. The conjugation is essential for the chosen coefficient convention and is not a foundational obstacle.

The separation proof, however, is phrased classically as a contradiction: not all coefficients can have equal values at distinct points. Constructively, a positive separation statement asks for a representation and coefficient whose values are apart. Negating a universal equality is not automatically a construction of such a witness. Nor is classical $a\ne b$ the same datum as a positive metric separation in a constructive formulation.

The finite approximation method explains how a quantitative witness can be sought when the requisite data are available. Suppose $g$, $h$, and a positive number
\[
c=\|L_gh-h\|_2>0
\]
are supplied. Approximate $h$ by $v$ in a finite invariant continuous space with $\|h-v\|_2<c/4$. Unitarity then gives
\[
\|L_gv-v\|_2\ge c-2\|h-v\|_2>c/2.
\]
Thus $L_g$ is detectably nontrivial on that finite space. With an explicit orthonormal basis, its matrix has an entry apart from the corresponding identity entry, which supplies a separating coefficient. This describes the required positive data; it is not a choice-free proof for arbitrary compact Hausdorff groups.

Uniform density is yet another step. Its classical Stone--Weierstrass proof uses the coefficient algebra, constants, conjugation, and point separation. Constructive Stone--Weierstrass results exist \cite{BanaschewskiMulveySW}; applying one requires matching its setting and separation hypotheses. The uniform conclusion here is not a disguised consequence of the previous $L^2$ estimate.

\subsection{Step 6: irreducibles and orthogonality}

\begin{lemma}[Complete reducibility]\label{lem:complete-reducibility}
Every finite-dimensional unitary representation is an orthogonal direct sum of irreducibles.
\end{lemma}
\begin{proof}
Induct on the dimension. If the representation is irreducible, there is nothing to prove. Otherwise choose a nonzero proper invariant subspace $W$. Unitarity makes $W^\perp$ invariant, and induction applies to both spaces, whose dimensions are strictly smaller.
\end{proof}

\begin{lemma}[Schur orthogonality]\label{lem:schur-orthogonality}
For the chosen irreducible representatives, write $\pi_{ij}(x)=\langle\pi(x)e_j,e_i\rangle$. Then
\[
\int_G\pi_{ij}(x)\overline{\sigma_{k\ell}(x)}\,d\mu(x)
=\begin{cases}
d_\pi^{-1}\delta_{ik}\delta_{j\ell},&\pi=\sigma,\\
0,&\pi\not\cong\sigma.
\end{cases}
\]
\end{lemma}
\begin{proof}
For $T:H_\pi\to H_\sigma$, the operator
\[
A_T=\int_G\sigma(x)T\pi(x)^{-1}\,d\mu(x)
\]
intertwines $\pi$ and $\sigma$, by left invariance: substituting $y=gx$ gives $\sigma(g)A_T\pi(g)^{-1}=A_T$.

We recall the classical Schur argument to specify the ingredient being used. The kernel and image of an intertwiner are invariant. For irreducibles, a nonzero intertwiner must therefore have zero kernel and full image, hence be an isomorphism. Thus no nonzero intertwiner exists between inequivalent irreducibles. For an endomorphism $A$ of an irreducible complex representation, choose an eigenvalue $\lambda$ by the classical existence of a root of its characteristic polynomial. The nonzero invariant kernel of $A-\lambda I$ must be the whole space, so $A=\lambda I$.

It follows that $A_T=0$ for inequivalent irreducibles. If $\pi=\sigma$, then $A_T$ is scalar. Since trace is unchanged under conjugation and $\mu(G)=1$, $\operatorname{Tr}A_T=\operatorname{Tr}T$, and hence $A_T=(\operatorname{Tr}T/d_\pi)I$. For the rank-one map $Te_i=e_k$ and $Te_r=0$ for $r\ne i$, its $(\ell,j)$ entry gives
\[
\int_G\sigma_{\ell k}(x)\overline{\pi_{ji}(x)}\,d\mu(x)
=\frac{\delta_{ki}\delta_{\ell j}}{d_\pi}
\]
in the equal-representation case, and zero otherwise. Taking conjugates and relabeling indices yields the displayed formula.
\end{proof}

\begin{proof}[Proof of \Cref{thm:classical-PW}]
Part (1) is \Cref{lem:haar-average}; Part (2) is \Cref{prop:uniform-density}. By \Cref{prop:L2-density,lem:invariant-coefficients}, coefficients of finite-dimensional representations span a dense subspace of $L^2(G)$. Complete reducibility reduces this span to coefficients of the chosen irreducible representatives, since equivalent representations have the same coefficient space. Schur orthogonality supplies their normalization and orthogonality, proving Part (3).

For Part (4), identify $H_\pi^*$ with its conjugate Hilbert space by the inner product. The linear map
\[
H_\pi\otimes H_\pi^*\longrightarrow\mathcal C_\pi,
\qquad v\otimes\langle\,\cdot\,,w\rangle
\longmapsto\sqrt{d_\pi}\,\varphi_{v,w}
\]
is unitary by Schur orthogonality. The orthonormal basis in Part (3) therefore gives the Hilbert direct sum, including the dimension assertion.
\end{proof}

\paragraph{Constructive comparison: irreducibility and orthogonality.}
Complete reducibility above uses the classical alternative: either a representation is irreducible or it has a nonzero proper invariant subspace. The constructive issue is both the alternative and the data it supplies. To form $W^\perp$ and continue an explicit splitting, one needs a suitable basis or located orthogonal projection for $W$, together with its dimension and invariance. A subspace named by an existence statement does not automatically provide these data. The same warning applies to choosing one representative and a basis for every irreducible equivalence class.

After appropriate finite representation data and intertwiner conclusions are established, the Haar averaging calculation for $A_T$ and the trace normalization in Schur orthogonality are concrete. The division by $d_\pi$ is harmless for a supplied positive integer dimension. What must be checked is the irreducibility and intertwiner argument that makes $A_T$ zero or scalar; the integral identity alone does not produce a classification of all irreducibles. Consequently, the constructive central spectral blocks of the main text should not simply be called irreducible representations. Classically a block has dimension $d_\pi^2$ and contains $d_\pi$ copies of $\pi^*$ under the left action. The blockwise Hilbert-sum result is already established under the main-text assumptions, whereas a finer irreducible description is a separate comparison.

\subsection{How the comparison supports the status assessment}\label{app:comparison}

The proof makes several different kinds of passage visible.
\begin{enumerate}
\item \textbf{Calculations with supplied data.} Haar averaging, changes of variables, convolution estimates, finite-basis coefficient formulas, and finite-dimensional orthogonality calculations can retain their form once integration, positivity, and representation data are justified.
\item \textbf{Classical existence steps.} Compact Hausdorff bump functions, spectral eigenspaces and their global range identity, contradiction-based separation, and invariant-subspace selections provide data automatically in the classical proof. Each requires a specified constructive theorem or a direct positive construction when transferred.
\item \textbf{A reorganized constructive proof.} The normable central Gelfand input, normalized convolution idempotents, and finite-tail estimates replace the appendix's route through the classical compact-operator spectral theorem. The result is the established approximation core of \Cref{thm:CS-schematic}; its blockwise Hilbert completion follows as well.
\item \textbf{Different conclusions and settings.} Strong $L^2$ approximation, uniform coefficient density, a decomposition into spectral blocks, and an irreducible decomposition must not be conflated. Compact Hausdorff topology, Bishop compact metric data, and point-free compactness also have different assumptions. None of these differences alone certifies an open problem.
\end{enumerate}

Thus there is no single algebraic line at which ``the classical proof becomes nonconstructive.'' Some ingredients require a stronger input from the outset; others admit direct quantitative versions; the spectral and representation-theoretic choices may require a different organization. This is why the appendix is included in a survey of what is known and what remains open: it identifies the content to compare, while the main text and its cited constructive results determine what has actually been proved. The comparison does not turn the classical proof into a constructive one, and it does not claim that every classical existence step lacks a constructive counterpart.

\end{document}